\documentclass[12pt]{article}

\usepackage{setspace}
\usepackage{a4wide}

\usepackage{amsthm, amssymb, amstext}
\usepackage[fleqn]{amsmath}
\usepackage{latexsym}
\usepackage[dvips]{graphicx}
\usepackage{comment}
\usepackage{hyperref}
\usepackage{mathtools}
\usepackage{enumerate} 

\usepackage{todonotes}
\usepackage{comment}

\newtheorem{theorem}{Theorem}
\newtheorem{lemma}{Lemma}
\newtheorem{proposition}{Proposition}

\newtheorem{conjecture}{Conjecture}

\newtheorem{claim}{Claim}

\theoremstyle{definition}

\usepackage{amsthm}

\title{Kempe equivalence of $4$-critical planar graphs}
\author{Carl Feghali\thanks{Univ Lyon, EnsL, CNRS, LIP, F-69342, Lyon Cedex 07, France, email: \texttt{carl.feghali@ens-lyon.fr}}}

\date{}

\begin{document}
\maketitle

\begin{abstract}
Answering a question of Mohar from 2007, we show that for every $4$-critical planar graph, its set of $4$-colorings is a Kempe class. 
 \end{abstract}

\section{Introduction}

Let $G$ be a graph, and let $k$ be a non-negative integer. 
A (proper) \emph{$k$-coloring} of $G$ is a function $\varphi: V(G) \rightarrow \{1, \dots, k\}$ such that $\varphi(u) \neq \varphi(v)$ whenever $uv \in E(G)$.

 A \emph{Kempe chain} in colors $\{a, b\}$ is a maximal connected subgraph $K$ of $G$ such that every vertex of $K$ has color $a$ or $b$. 
 By swapping the colors $a$ and $b$ on $K$, a new coloring is obtained. This operation is called a \emph{$K$-change}. If $c_2$ is a $k$-coloring obtained from a $k$-coloring $c_1$ by a single $K$-change, then we write $c_1 \sim_k c_2$. Two $k$-colorings $c_1$ and $c_2$ are \emph{$K$-equivalent} (or $K^k$-equivalent) if $c_1$ be obtained from $c_2$ by a sequence of $K$-changes. A $K$-change is \emph{trivial} if the Kempe chain under consideration consists of a single vertex. 
 
 Let $C_G^k$ be the set of $k$-colorings of $G$. The equivalence classes $C_G^k / \sim_k$ are called \emph{Kempe classes}. The number of Kempe classes of $G$ is denoted $\mbox{Kc}(G, k)$. 
 
 Kempe chains were introduced by Kempe in his failed attempt at proving the Four Colour Theorem. Nevertheless, they have proved to be one of the most useful tools in graph coloring theory. As pointed out by Mohar \cite{mohar1}, the number of Kempe equivalence classes of colorings has some applications in statistical physics \cite{mohar2009new} and Markov chains \cite{vigoda2000improved}. Furthermore, there are many results on determining whether all colorings of a graph are Kempe equivalent or not whenever the graph belongs to a special graph class. For instance, Meyniel \cite{meyniel1} showed that the set of $5$-colorings of a planar graph forms a single Kempe class, and Las Vergnas and Meyniel \cite{meyniel3} extended this result to $K_5$-minor-free graphs. Mohar \cite{mohar1} showed that the set of $4$-colorings of every $3$-colorable planar graph is a Kempe class, and asked whether his result can be extended to $4$-critical planar graphs (a graph $G$ is \emph{$4$-critical} if it is not $3$-colorable, but every proper subgraph of $G$ is $3$-colorable). In that same paper, Mohar also conjectured that the set of $k$-colorings of a $k$-regular graph forms a Kempe class, and this was settled in \cite{bonamyx, FJP15}. For further details and examples, we refer the reader to \cite{bonamydiameter, bonamyshortest, cranston2021kempe, mohar1, He13}. The edge-coloring version has also been considered in several papers \cite{belcastrocounting, bonamy2021vizing, mcdonald, ozeki2022kempe} so has the vertex-coloring version where only trivial Kempe changes are allowed \cite{cereceda2008connectedness, feghali2016reconfigurations}.
 
 In this note, we answer Mohar's aforementioned problem in the positive. 

\begin{theorem}\label{thm:main}
Let $G$ be a $4$-critical planar graph. Then $\mbox{Kc}(G, 4) = 1$. 
\end{theorem}

We should note that Theorem \ref{thm:main} is best possible in the sense that there are infinitely many planar graphs with more than one Kempe equivalence class as shown by Mohar \cite{moharakempic}. 

\section{Preliminaries}

In this section, we gather some of the necessary tools to establish Theorem~\ref{thm:main}. The first lemma is implicit in \cite{mohar1}.

\begin{lemma}[Mohar~{\cite[Proposition 2.4]{mohar1}}]\label{lem:three}
Let $k \geq 1$ be an integer,  $G$ be a graph, and $v$ a vertex of $G$ with degree strictly less than $k$. If $\mbox{Kc}(G - v, k) = 1$, then $\mbox{Kc}(G, k) = 1$. 
\end{lemma}

\begin{lemma}[Mohar~{\cite[Lemma 4.2]{mohar1}}]\label{lem:super}
Suppose that $G$ is a subgraph of a graph $G'$. Let $c'_1, c'_2$ be $r$-colorings of $G'$. For $i \in \{1, 2\}$, denote by $c_i$ the restriction of $c'_i$ to $G$. If $c'_1$ and $c'_2$ are $K^r$-equivalent, then $c_1$ and $c_2$ are $K^r$-equivalent colorings of $G$. 
\end{lemma}

\begin{lemma}[Mohar~{\cite[Theorem 4.4]{mohar1}}]\label{lem:mohar}
For every $3$-colorable planar graph $G$, $\mbox{Kc}(G, 4) = 1$. 
\end{lemma}

Let $G$ be a planar graph, and let $S \subseteq V(G)$. Two $4$-colorings $c_1$, $c_2$ of $G$ are said to be \emph{$K_S$-equivalent} if there exists a sequence of $K$-changes from $c_1$ to $c_2$ that does not change the color of a vertex in $S$ to color $4$. We have the following proposition. 

\begin{proposition}\label{prop1}
Let $G = (V, E)$ be a $3$-colorable triangulation of the plane. Then every $4$-coloring  of $G$ is $K_V$-equivalent to the (unique) $3$-coloring of $G$. 
\end{proposition}

The proof of Proposition \ref{prop1} follows, nearly word by word, from the proof of Theorem 1 in \cite{fisk} by Fisk. We include some of the details for the convenience of the reader. First, we require some definitions and an auxiliary lemma. 

Let $f$ be a $4$-coloring of a triangulation $G$, and let $e =xy$ be an edge of $G$. By the color of $e$ under $f$, we mean $\{f(x), f(y)\}$.   Two triangles $xyz$ and $xyw$ contain the edge $e$. If $f(w) = f(z)$ then $e$ is called \emph{singular}, and if $f(w) \not=f(z)$ then $e$ is \emph{nonsingular}. We require the following result of Fisk extracted from the proof of Theorem 1 in \cite{fisk}. 

\begin{lemma}[Fisk~\cite{fisk}]\label{lem:fisk}
Let $G$ be a $3$-colorable triangulation of the plane, and let $f$ be a $4$-coloring of $G$. If $G$ contains a nonsingular edge $e$ under $f$ and $S$ denotes the set of nonsingular edges of $G$ having the same color as $e$, then $S$ contains a cycle that bounds some region of the plane.  
\end{lemma}

\begin{proof}[Proof of Proposition \ref{prop1}]
Given a $4$-coloring $f$ of $G$, we show how to obtain a new coloring $g$ that is $K_V$-equivalent to $f$ and with fewer nonsingular edges. Note that each edge in the unique $3$-coloring of $G$ is singular, and thus, by repetition, this will prove the proposition. 

We first show that for every $v \in V(G)$ with $f(v) = 4$, either $v$ is incident with a nonsingular edge or we may change the color of $v$ to a color distinct from $4$ by exhibiting a trivial K-change. To see this, note that since $G$ is a triangulation and is $3$-colorable, the neighborhood of $v$ induces an even cycle $C$. If $|f(C)| = 2$, i.e., there is a color $a \in \{1, 2, 3\}$ not used by $f$ on $C$, then we may change the color of $v$ to $a$. Otherwise, $f(C) = \{1, 2, 3\}$ and, by the pigeonhole principle, some edge incident with $v$ is nonsingular, as needed. 

To prove the proposition, we can thus safely assume that there is a nonsingular edge $e$ under $f$ having color $\{a, 4\}$ for some $a \in \{1, 2, 3\}$. By Lemma \ref{lem:fisk}, there is a cycle in $G$ whose edges have the same color as $e$ and bounding some region $D$ of the plane.  
By interchanging the two colors in $\{1,2, 3\} \setminus \{a\}$ in  the interior of $D$, we obtain a new coloring $g$ with fewer nonsingular edges than $f$ (the edges in the interior of $D$ stay singular or nonsingular, while the edges of the cycle on the boundary of $D$ change from nonsingular to singular). This completes the proof. 
\end{proof}

We will also need the following result of Mohar. A plane graph is a \emph{near-triangulation} if every facial cycle of the graph is a triangle except possibly for the outer cycle. 

\begin{proposition}[Mohar~{\cite[Proposition 4.3]{mohar1}}]\label{lem:34}
Suppose that $G$ is a planar graph with a facial cycle $C$. If $c_1$, $c_2$ are $4$-colorings of $G$, then there is a near-triangulation $T$ of the plane with the outer cycle $C$ such that $T \cap G = C$ and there are $4$-colorings $c'_1$, $c'_2$ of $G$ that are $K$-equivalent to $c_1$ and $c_2$, respectively, such that $c'_1$ and $c'_2$ can be extended to $4$-colorings of $T \cup G$. Moreover, if the restriction of $c_1$ to $C$ is a $3$-coloring, then $c'_1 = c_1$, and $c_1$ can be extended to a $3$-coloring of $T \cup G$.  
\end{proposition}

\section{The proof of Theorem \ref{thm:main}}

In this section, we prove Theorem \ref{thm:main}. We begin with some lemmas.

\begin{lemma}[Koester~\cite{koester}]\label{lem:four}
Every $4$-critical planar graph has a vertex of degree at most~$4$. 
\end{lemma}

For a graph $G$, a subgraph $H$ of $G$ and a coloring $\varphi$ of $G$, let $\varphi \restriction H$ denote the restriction of $\varphi$ to $H$. For a vertex $v$ of $G$ of degree exactly $4$ and a $4$-coloring $f$ of $G$, we say that $f$ is \emph{$v$-good} if $f$ colors exactly two neighbors of $v$ alike; in other words, if all four colors appear in the closed neighborhood of $v$ under $f$.

\begin{lemma}\label{lem:good}
Let $G$ be a $4$-critical planar graph, let $v$ be a vertex of degree exactly four in $G$, and let $\varphi$ be a $4$-coloring of $G$. Then $\varphi$ is $K$-equivalent to a $v$-good $4$-coloring $\psi$ of $G$. 
\end{lemma}

\begin{proof}
Since $G$ is $4$-critical, the graph $H = G - v$ is $3$-colorable. By Lemma \ref{lem:mohar}, there is a sequence of $K$-changes $\varphi \restriction H = c'_1 \sim_4 \dots \sim_4 c'_m$ in $H$ where $c'_m$ is a $3$-coloring; moreover, $c'_m$ uses all three colors in the neighborhood of $v$ else $G$ would be $3$-colorable. 

Thus, we may choose $t \in \{1, \dots, m\}$ to be the smallest index such that $c'_t$ colors exactly two neighbors of $v$ alike. Now, the same argument as in the proof of Lemma \ref{lem:three} can be applied to show that the sequence of $K$-changes $\varphi \restriction H = c'_1 \sim_4 \dots \sim_4 c'_t$ extends to a sequence of $K$-changes from $\varphi$ to some $4$-coloring $\varphi'$ of $G$ such that $\varphi' \restriction H = c'_t$ and so $\psi = \varphi'$  is our required coloring. We repeat the argument here for completeness. 

For $j \in \{1, \dots, t\}$, we extend $c_j'$ to a $4$-coloring $c_j$ of $G$ as follows. Observe that each $c'_{i + 1}$ differs from $c'_i$ by a single $K$-change. We then use the same $K$-change in $G$ to obtain $c_{i+1}$ from $c_i$, unless this $K$-change involves $v$ and there are at least two neighbors of $v$ in the component under consideration. However, by hypothesis, $v$ has at least three neighbors colored alike and has degree $4$ and so there is a color $a$ not appearing in its closed neighborhood. Thus, by first preceding this $K$-change by changing the color of $v$ to $a$ (via a trivial $K$-change), the result follows.   
\end{proof}

\begin{lemma}\label{lem:final}
Let $G$ be a $4$-critical plane graph. Let $v$ be a vertex of degree exactly four in $G$. Let $c_1, c_2$ be $v$-good $4$-colorings of $G$ such that $c_1 \restriction (G - v)$ is a $3$-coloring. Then $c_2$ is $K$-equivalent to $c_1$. 
\end{lemma}

\begin{proof}
We may assume that $v$ is on the outer face of $G$. Let $v_0, \dots, v_3$ denote the neighbors of $v$ in this clockwise order around $v$. We begin by constructing a new graph $G^*$ from $G$ as follows. 

 Initially, we set $G^* = G$. For $i \in \{0,1, 2, 3\}$, if $v_iv_{i + 1} \not\in E(G)$ (where addition is taken modulo $4$) and $c_j(v_i) \not= c_j(v_{i + 1})$ for each $j \in \{1, 2\}$, then we add the edge $v_iv_{i+1}$ to $G^*$. If $v_iv_{i + 1} \not\in E(G)$ but $c_j(v_i) = c_j(v_{i + 1})$ for some $j \in \{1, 2\}$, then we add to $G^*$ a new vertex $u_{i, i+1}$ adjacent to $v_i, v_{i + 1}, v$. 
 
 It is not difficult to see that $G^*$ is planar, and that $c_1, c_2$ extend to $4$-colorings $c_1^*, c_2^*$ of $G^*$ such that $c_1^* \restriction (G^* - v)$ is a $3$-coloring; moreover, if $v_iv_{i + 1} \not\in E(G)$ and $c_1(v_i) = c_1(v_{i + 1})$, then, by the $v$-goodness of $c_2$, there is a vertex $z$  where either $z = v_{i-1}$ and $zv_{i} \in E(G^*)$ or $z = v_{i+2}$  and $zv_{i + 1} \in E(G^*)$. In this case, we set 
 \begin{equation}\label{eq:1}
 c^*_1(u_{i, i+1}) = c^*_1(z).
 \end{equation}   
 
\begin{claim}
There is a near-triangulation $M$ of the plane containing $G^*$ with a $4$-coloring $f$ such that $f \restriction (M - v)$ is a $3$-coloring and $c_1^*$ is $K$-equivalent to $f \restriction G^*$.
\end{claim}

\begin{proof}

 Let $C_1, \dots, C_m$ be the facial cycles of $G^*$ not incident with $v$. Let $G_0 = G^*$, $c_1^0 = c_1^*$ and $c_2^0 = c_2^*$. For $i \in \{1, \dots, m\}$ we apply Proposition \ref{lem:34} to the graph $G_{i - 1}$, its facial cycle $C_i$  and the colorings $c_1^{i-1}$ and $c_2^{i-1}$. We conclude that there is a near-triangulation $T_i$ with outer cycle $C_i$ such that $G_{i-1} \cap T_i = C_i$. Let $G_i = G_{i-1} \cup T_i$. By Proposition \ref{lem:34}, $c_1^{i-1}$ can be extended to a $3$-coloring $c_1^i$ of $G_i$ and $c_2^{i-1}$ is $K$-equivalent in $G_{i-1}$ to a $4$-coloring that has an extension $c_2^i$ to $G_i$. The final graph $G_m$ is a near-triangulation of the plane, with the $4$-coloring $c_1^m$ such that $c_1^m \restriction (G_m - v)$ is a $3$-coloring, as claimed.
 \end{proof}

Borrowing the notation from the proof of the claim,  let $H = G_m - v$, and assume up to a single $K$-change that $c_2^m(v) = 4$. Note that $H$ is a near-triangulation with the outer face bounded by the cycle $C$ consisting of the neighbors of $v$ in $G_m$.  For $i \in \{1, 2\}$, let $c'_i = c_i^m \restriction H$.

\begin{claim} There is a triangulation $N$ of the plane  containing $H$ such that $c_1'$ extends to a $3$-coloring $c_1^N$ of $N$ and $c_2'$ to a $4$-coloring $c_2^N$ of $N$. 
\end{claim}

\begin{figure}
\begin{center}\includegraphics[width=0.4\textwidth]{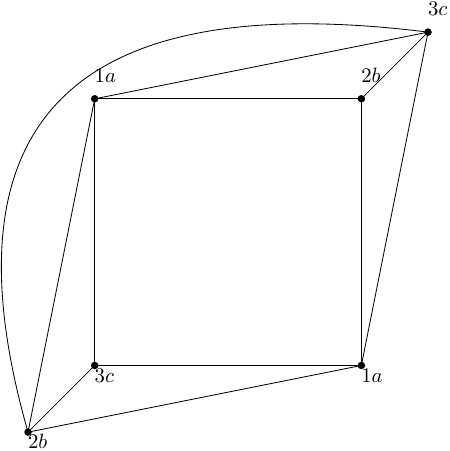}\end{center}
\caption{The first case}\label{fig:case1}
\end{figure}

\begin{figure}
\begin{center}\includegraphics[width=0.3\textwidth]{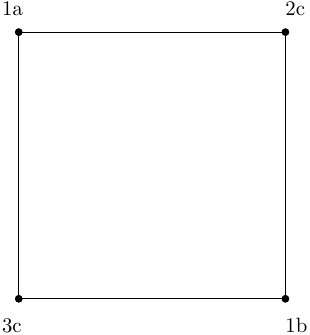}\end{center}
\caption{The special case (up to permutations of colors)}\label{fig:bad}
\end{figure}

\begin{figure}
\begin{center}\includegraphics[width=0.8\textwidth]{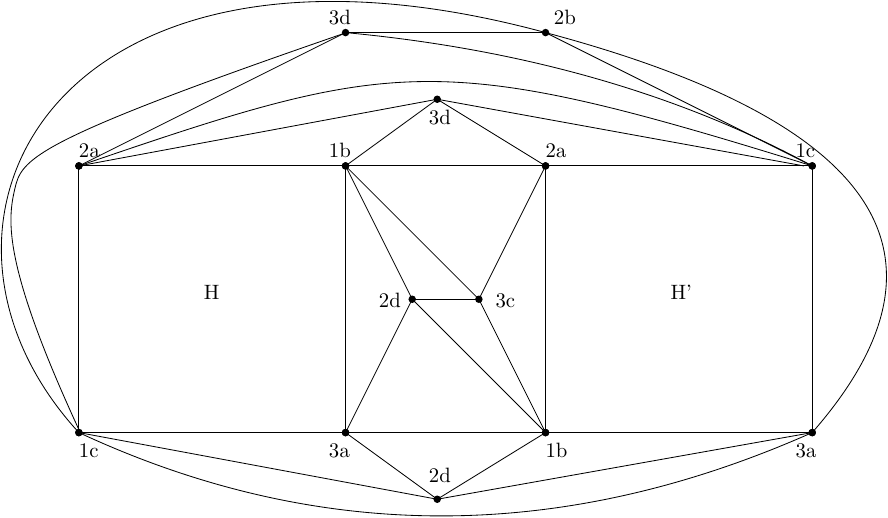}\end{center}
\caption{The construction}
\label{fig:fix}
\end{figure}
 
 \begin{proof}Observe by the construction of $G^*$ and the definition of $v$-goodness that $C = x_1 \dots x_{k}x_0$ for some $k \in \{4, 5, 6\}$. We distinguish two cases.
 
 \medskip
 
 \textbf{Case 1:} $k = 4$. Without loss of generality,  we have $c'_1(x_1) = c'_1(x_3)$, and the goodness of $c_1$ implies $c'_1(x_2) \not= c'_1(x_4)$. We distinguish two subcases.
 
\medskip

\textbf{Subcase 1.1:} $c'_2(x_1) = c'_2(x_3)$. In this case, we add two exterior vertices $w_1, w_2$ where $w_1$ is adjacent to $w_2, x_1, x_3, x_4$ and $w_2$ is adjacent to $w_1, x_1, x_2, x_3$; see Figure \ref{fig:case1} (where the coloring $c'_1$ is represented by colors 1--3 and the coloring $c'_2$ is represented by colors a--c).  

\medskip

\textbf{Subcase 1.2:} $c'_2(x_1) \not=c'_2(x_3)$. In this case, we arrive at the configuration in Figure \ref{fig:bad}. We take a copy $H'$ of $H$ (where the vertices with colors $1b$ and $1c$ are flipped) and add edges and new vertices as shown in Figure \ref{fig:fix}. This completes the case $k = 4$.

\medskip

\textbf{Case 2:} $k \in \{5, 6\}$. Following Proposition 4.3 in \cite{mohar1} verbatim, we show that there are non-consecutive vertices $x_i, x_j$ on $C$ such that $c'_1(x_i) \not=  c'_1(x_j)$ and $c'_2(x_i) \not= c'_2(x_j)$.  We may assume that $c'_1(x_1) = c'_1(x_3)$. Then $c'_1(x_1) \not=c'_1(x_4)$ so $c'_2(x_1) = c'_2(x_4) = a$ . Suppose that $c'_1(x_2) = 2$ and $c'_2(x_2) = b$. Since $b = c'_2(x_2) \not= c'_2(x_4) = a$, we have $c'_1(x_4) = c'_1(x_2) = 2$. Now, $c'_1(x_2) \not= c'_1(x_5)$, so $c'_2(x_5) = b$. Next, $c'_2(x_5) \not= c'_2(x_3)$ implies that $c'_1(x_5) = c'_1(x_3) = 1$, but then $x_1$ and $x_5$ are colored alike under $c_1'$ and so cannot be adjacent. It follows, in particular, that $k = 6$. Similar conclusions as before imply that $c'_1(x_6) = 2$ and $c'_2(x_3) = c'_2(x_6) = c$. But then $c_1$ is not  a good $4$-coloring of $G$ by (\ref{eq:1}) and our aim is achieved. 

We proceed to add the edge $x_ix_j$ outside $C$.  If $k = 5$ then $x_j = x_{i+2}$ and we let the cycle $C' = x_ix_{i+2}x_{i+3}x_{i+4}x_i$ play the role of $C$ in Case 1. Similarly, if $k = 6$ and $x_j = x_{i+2}$, we let the cycle $C' = x_ix_{i+2}x_{i+3}x_{i+4} x_{i + 5}x_i$ play the role of $C'$ in the preceding sentence.

We are left to consider the case $k = 6$ with  $i = 1$ and $j = 4$. We distinguish two cases.

\medskip

\textbf{Subcase 2.1:} the restrictions of $c'_1$ and $c'_2$ to the cycle $C^1=x_1x_2x_3x_4x_1$ or the cycle $C^2 = x_1x_6x_5x_4x_1$ is not as in Figure \ref{fig:bad}. In this case, we apply the argument from Case 1 simultaneously to $C^1$ and $C^2$. Precisely, if both are not as in Figure \ref{fig:bad}, then we apply the construction of Subcase 1.1 for one cycle by placing the two new vertices in its interior and then we apply the construction of Subcase 1.1 for the other cycle by placing the two new vertices in its exterior; similarly, if only of them is not as in Figure \ref{fig:bad}, then we apply the construction of Subcase 1.1 for this cycle by placing the two new vertices in its interior and then we apply the construction of Subcase 1.2 for the other cycle. The resulting graph satisfies the claim and Subcase 2.1 is completed.

\medskip

\textbf{Subcase 2.2:} the restrictions of $c'_1$ and $c'_2$ to both $C^1$ and $C^2$ are as in Figure \ref{fig:bad}. In this case, we may assume $c'_1(x_1) = 1 \not=c'_1(x_4) = 2$ and $c'_2(x_1) = a \not= c'_2(x_4) = b$. We may also assume that $c'_1(x_3) = c'_1(x_1)= 1$, $c'_2(x_4) = c'_2(x_2) = b$, $c'_1(x_2) = 3$ and $c'_2(x_3) = c$.  For the colorings of $C^2$, there are two possibilities meeting the requirements of Figure \ref{fig:bad}. 

The first possibility is $c'_1(x_5) = c'_1(x_1) = 1$, $c'_2(x_6) = c'_2(x_4) = b$, $c'_1(x_6) = 3$ and $c'_2(x_5) = c$. In this case, we delete the edge $x_1x_4$ and add instead the edge $x_3x_6$ outside $C$. The edge $x_3x_6$ splits $C$ into two parts which do not both match Figure \ref{fig:bad}, and hence we can apply the procedure as in the treatment of Subcase 2.1. 

The second and last possibility is $c'_2(x_5) = c'_2(x_1) = a$, $c'_1(x_6) = c'_1(x_4) = 2$, $c'_1(x_5) = 3$ and $c'_2(x_6) = c$. In this case, we delete the edge $x_1x_4$ and add instead the edge $x_2x_6$ outside $C$ and proceed as before. The claim is proved. 
\end{proof}

 We can now complete the proof of the lemma. By Lemma \ref{lem:fisk}, $c_2^N$ is $K_{V(N)}$-equivalent to $c_1^N$. By Lemma \ref{lem:super}, $c'_2$ is $K_{V(H)}$-equivalent to $c'_1$. Since $c_2^m(v) = c_1^m(v) = 4$, it follows vacuously that $c_2^m$ is $K$-equivalent to $c_1^m$. By successively applying Lemma \ref{lem:super}  to $G_{m - 1} \subseteq G_m$ etc. up until $G_0 \subseteq G_1$ we conclude that $c_2^{m-1}$ is $K$-equivalent to $c_1^{m-1}$ in $G_{m-1}$, etc. until finally concluding that $c_1^* = c_1^0$ is $K$-equivalent to $c_2^0 = c_2^*$ in $G^*$. By Lemma \ref{lem:super}, $c_2$ is $K$-equivalent to $c_1$. The lemma is proved. 
\end{proof}

We are now ready to prove Theorem \ref{thm:main}. 

\begin{proof}[Proof of Theorem \ref{thm:main}]
Suppose $G$ has a vertex $v$ of degree three or less. Since $G$ is $4$-critical, $G - v$ is $3$-colorable.  By Lemma \ref{lem:mohar}, $\mbox{Kc}(G - v, 4) = 1$. By Lemma \ref{lem:three}, $\mbox{Kc}(G, 4) = 1$. So we can assume that every vertex of $G$ has degree at least four. By Lemma \ref{lem:four}, $G$ has a vertex $v$ of degree exactly four. Let $\varphi_1$ and $\varphi_2$ be $4$-coloring of $G$. To prove the theorem, it suffices to show that $\varphi_1$ and $\varphi_2$ are $K$-equivalent.

By Lemma \ref{lem:good}, we may assume that $\varphi_1$ and $\varphi_2$ are $v$-good $4$-colorings of $G$. Since $G$ is $4$-critical, there is a $v$-good $4$-coloring $\psi$ of $G$ such that $\psi \restriction (G - v)$ is a $3$-coloring. By Lemma \ref{lem:final}, for $i \in \{1, 2\}$ $\varphi_i$ is $K$-equivalent to $\psi$.  Thus, $\varphi_1$ is $K$-equivalent to $\varphi_2$ and the theorem is proved. 
\end{proof}

\section{Final remarks}

Meyniel  \cite{meyniel1} showed that the set of $5$-colorings of a planar graph is a Kempe class. We propose the following generalization. A list
assignment of a graph is a function $L$ that assigns to each vertex $v$ a list $L(v)$ of
colors. The graph $G$ is $L$-colorable if it has a proper coloring $f$ such that $f(v) \in L(v)$ for each vertex $v$ of $G$.

\begin{conjecture}\label{conj}
Let $G$ be a planar graph, and $L$ be a list assignment for $G$ such that $|L(v)| \geq 5$ for each $v \in V(G)$. Then the set of $L$-colorings of $G$ forms a Kempe class. 
\end{conjecture}

Note that the set of $L$-colorings of $G$ in the statement of Conjecture \ref{conj} is non-empty by a celebrated result of Thomassen \cite{thomassen1}. 

A fundamental theorem of Gr\"otzsch states that triangle-free planar graphs are $3$-colorable. Towards a Kempe analogue Gr\"otzsch's theorem, Salas and Sokal \cite{salas2022ergodicity} recently asked if $Kc(G, 3) = 1$ for every triangle-free planar graph $G$. We believe this to be true. 

\begin{conjecture}
If $G$ is a triangle-free planar graph, then $Kc(G, 3) = 1$.
\end{conjecture}

There are triangle-free planar graphs $G$ and a list assignment $L$ for $G$ such that $|L(v)| = 3$ for each $v \in V(G)$ with the property that $G$ is not $L$-colorable. However, if $G$ has girth at least $5$, then $G$ is always $L$-colorable \cite{thomassen2003short}. 

\begin{conjecture}
Let $G$ be a planar graph of girth $5$, and $L$ a list assignment for $G$ such that $|L(v)| \geq 3$ for each $v \in V(G)$. Then the set of $L$-colorings of $G$ forms a Kempe class. 
\end{conjecture}

\section*{Acknowledgments}

I am very grateful to Mykhaylo Tyomkyn for several helpful discussions. This work was partially supported by grant 19-21082S of the Czech Science Foundation and the French National Research Agency under research grant ANR DIGRAPHS ANR-19-CE48-0013-01

 \bibliography{bibliography}{}
\bibliographystyle{abbrv}
 
\end{document}